\documentclass[14pt]{amsart}
\usepackage{cases}
\usepackage{amsmath}
\usepackage{amsfonts}
\usepackage{bm}
\usepackage{amsfonts,amsmath,amssymb,amscd,bbm,amsthm,mathrsfs,dsfont}
\usepackage{mathrsfs}
\usepackage{pb-diagram}
\usepackage{color}
\usepackage{amssymb}
\usepackage{xypic}

\usepackage{mathrsfs}
\usepackage{amsthm}
\usepackage[all]{xy}
\usepackage{epsfig}
\usepackage{graphics}
\usepackage{array}
\usepackage{graphicx} 
\usepackage{epstopdf}
\usepackage{float}

\newtheorem{Theorem}{Theorem}[section]
\newtheorem{Lemma}[Theorem]{Lemma}
\newtheorem{Definition}[Theorem]{Definition}
\newtheorem{Corollary}[Theorem]{Corollary}
\newtheorem{Proposition}[Theorem]{Proposition}
\newtheorem{Example}[Theorem]{Example}
\newtheorem{Remark}[Theorem]{Remark}
\newtheorem{Conjecture}[Theorem]{Conjecture}

\title [On Exchange Spectra of Valued Cluster Quivers]{On Exchange Spectra of Valued Cluster Quivers \\and Cluster Algebras}

\author{Fang Li $\;\;\;\;\;\;$ Siyang Liu $\;\;\;\;\;\;$}
\address{Fang Li
\newline Department of Mathematics, Zhejiang University (Yuquan Campus), Hangzhou, Zhejiang
310027,  P.R.China}
\email{fangli@zju.edu.cn}

\address{Siyang Liu
\newline Department
of Mathematics, Zhejiang University (Yuquan Campus), Hangzhou, Zhejiang
310027, P.R.China}
\email{siyangliu@zju.edu.cn}

\date{version of \today}

\newcommand{\lra}{\longrightarrow}

\newcommand{\ra}{\rightarrow}
\newcommand{\sdp}{\times\kern-.2em\vrule height1.1ex depth-.05ex}
\newcommand{\epi}{\lra \kern-.8em\ra}

\makeatletter
\@addtoreset{equation}{section}
\makeatother

 \normalbaselines
\begin{document}

\renewcommand{\thefootnote}{\alph{footnote}}
\maketitle
\bigskip
\begin{abstract}

Inspirited by the importance of the spectral theory of graphs, we introduce the spectral theory of valued cluster quiver of a cluster algebra. Our aim is to characterize a cluster algebra via its spectrum so as to use the spectral theory as a tool.

First, we give the relations between exchange spectrum of a valued cluster quiver and adjacency spectrum of its underlying valued graph, and between exchange spectra of a valued cluster quiver and its full valued subquivers. The key point is to find some invariants from the spectrum theory under mutations of cluster algebras, which is the second part we discuss. We give a sufficient and necessary condition for a cluster quiver and its mutation to be cospectral. Following this discussion, the so-called cospectral subalgebra of a cluster algebra is introduced.
We study bounds of exchange spectrum radii of cluster quivers and give a characterization of $2$-maximal cluster quivers via the classification of oriented graphs of its mutation equivalence. Finally, as an application of this result, we obtain that the preprojective algebra of a cluster quiver of Dynkin type is representation-finite if and only if the cluster quiver is $2$-maximal.

\end{abstract}
\section{Introduction}
Cluster algebras were invented by Fomin and Zelevinsky in a series of papers \cite{FZ1,FZ2,BFZ,FZ4} and are thought to be a spectacular advance in mathematics. There are many relations and applications between cluster algebras and other important subjects, such as representations of quivers, combinatorics and quiver gauge theories.

In the theory of cluster algebras, two of vital roles are exchange matrices and mutation of them. Exchange matrices are assumed to be totally sign-skew-symmetric matrices introduced by Fomin and Zelevinsky in \cite{FZ1}. An important class of totally sign-skew-symmetric matrices consists of integer skew-symmetrizable matrices and they can be associated one-to-one correspondence with valued cluster quivers, which are simple oriented graphs without loops together with a pair of integers $(v(\alpha)_1,v(\alpha)_2)$ for each arrow $\alpha$ satisfying some rules. Skew-symmetrizable matrices and their mutations play very important roles in the study of cluster algebras, and many conjectures and problems are usually worked out firstly in this case. In \cite{FZ4}, Fomin and Zelevinsky conjectured that the exchange graph of a cluster algebra depends only on its initial exchange matrix, and it had been proved in the case of skew-symmetrizable matrices. So it is meaningful to study skew-symmetrizable matrices, their corresponding valued cluster quivers and mutation of them. Valued cluster quivers can be just regarded as valued oriented graphs and spectral graph theory is one of the major method to study properties of graphs.

In general, to reveal the properties of a graph, we usually associate a graph with some matrices and study these matrices via algebraic methods. In contrast, we can also make use of graph theory to study the properties of some matrices and transformations of matrices. Moreover, the spectral graph theory also has universal applications in many areas, such as information science, computer science, and communications. The most common matrices associated to (oriented) graphs are adjacency matrices and Laplacian matrices. There are lots of literature and results on spectral graph theory, especially for unoriented graphs, see e.g \cite{BH,C2}. But there are not enough attentions on spectral theory of oriented graphs. Recently, in \cite{C1} Chung considered the Laplacians for oriented graphs and studied their spectra, and Bauer introduced normalized Laplacians for weighted oriented graphs and investigated the properties of their spectra in \cite{B}. However, we are more interested in valued cluster quivers and their corresponding skew-symmetrizable matrices, hence we shall develop a novel spectral theory for valued cluster quivers in contrast to the classical spectral graph theory. On one hand, it can be contributed to study exchange matrices and their transformations. On the other hand, it support a new sight to study spectral graph theory.

This article is organized as follows. In section 2, some basic concepts and definitions are given and we characterize acyclicity of valued cluster quivers by its adjacency spectrum. In section 3, we investigate properties of spectra of exchange matrices of valued cluster quivers. Finally, in section 4 we study how mutations influence on exchange spectra of cluster quivers.

\section{Valued cluster quivers and exchange matrices }
\subsection{Definitions and notations}.

We follow \cite{ASS} for most basic concepts of quivers. A {\bf quiver} is an oriented graph described by a 4-tuple $Q = (Q_0,Q_1,s,t)$, where $Q_0$ is a set of {\bf vertices}, $Q_1$ is a set of {\bf arrows}, and $s,t$ are two functions that map each vertex to its {\bf source} and {\bf target}, respectively. We usually labeled the vertices by natural numbers. A path of length $p$ is a sequence of $p$ arrows $\alpha_1\alpha_2\dots \alpha_p$ satisfying that $s(\alpha_{j+1}) = t(\alpha_j),1 \leqslant j \leqslant p-1.$ For a path $\omega=\alpha_1\alpha_2\dots \alpha_p$, define $s(\omega)=s(\alpha_1)$ and $t(\omega)=t(\alpha_p)$.
A quiver $Q$ is said to be {\bf finite} if both $Q_0$ and $Q_1$ are finite sets, write $Q_0=\{1,\cdots,n\}$. In $Q$, if the multiplicities of arrows are at most $1$, then $Q$ is said to be {\bf simply-laced}. A {\bf sink} is a vertex $i \in Q_0$ satisfying that there is no arrow $\alpha \in Q_1$ such that $s(\alpha)=i$ and a {\bf source} is a vertex $j \in Q_0$ satisfying that there is no arrow $\alpha \in Q_1$ such that $t(\alpha)=j$.

  A {\bf full subquiver} $Q'$ of a quiver $Q$ is a quiver $Q' = (Q'_0,Q'_1,s',t')$ satisfying that $Q'_0\subseteq Q_0,  Q'_1\subseteq Q_1,  s' = s|_{Q'_1} , t'=t|_{Q'_1}$, and $Q'_1 = \{\gamma \in Q_1|s(\gamma),t(\gamma)\in Q'_0\}$.

The path algebra $KQ$ of a quiver $Q$ over an algebraically closed field $K$ is the $K$-vector space $KQ$ whose basis consisting of all paths in $Q$ with multiplication $\cdot$ defined on two basis elements $\omega_1$,$\omega_2$ by  \[\omega_1\cdot\omega_2=\begin{cases} \omega_1\omega_2,&\text{if $s(\omega_2)=t(\omega_1)$;}\\0,&\text{otherwise.}\end{cases}\]

The underlying graph of a quiver $Q$ is got by forgetting all orientations of arrows and is denoted by $\bar  Q$. We say a quiver $Q$ to be  {\bf connected} if its underlying graph is connected.

 A {\bf loop} of a quiver is just an arrow $\gamma$ such that $s(\gamma) = t(\gamma)$, and a $k$-cycle of a quiver is a path $\alpha_1\alpha_2\dots \alpha_k$ of length $k$ such that $s(\alpha_1) = t(\alpha_k)$. A chordless $k$-cycle in a quiver is a $k$-cycle such that no two vertices of the cycle are connected by an arrow that does not itself belong to the cycle.

  A {\bf cluster quiver} is a finite quiver without loops or 2-cycles.  For any quiver $Q$, the degree of any vertex is just its degree in the underlying graph $\bar Q$, ~i.~e. the number of edges incident with this vertex in $\bar Q$.

A {\bf valued cluster quiver} $(Q,v)$ is a finite quiver $Q$ without loops and at most one arrow between any pair of vertices, together with a {\bf valuing map} $v : Q_1 \rightarrow \mathbb{N}^2$ satisfying that there is a map $d : Q_0 \rightarrow \mathbb{N}_{>0}$ and for each arrow $\alpha: i\rightarrow j$ in $Q_1$, we have $d(i)v(\alpha)_1 = d(j)v(\alpha)_2$, where the {\bf value} $v(\alpha) = (v(\alpha)_1, v(\alpha)_2)$.

 By a little abuse of notation,  denote a valued cluster quiver $(Q,v)$ only by $Q$ and its underlying valued graph by $\bar{Q}$. For any arrow $\alpha: i\rightarrow j$,  the notation $(v_{ij},v_{ji})$ is used to replace $(v(\alpha)_1,v(\alpha)_2)$.

If $(Q,v)$ is a valued cluster quiver, $(Q',v')$ is called a {\bf full valued subquiver} of $Q$ if $Q'$ is a full subquiver of $Q$ and $v' = v|_{Q'_1}$. Note that $(Q',v')$ is also a valued cluster quiver.

For a valued cluster quiver $(Q,v)$, if $v(\alpha)_1 = 1 = v(\alpha)_2$ for any arrow $\alpha \in Q_1$, then we call $Q$ a {\bf simple cluster quiver}. Dealing with simple cluster quivers, we usually omit the labels.  Trivially,  simple cluster quivers are equivalent to simply-laced cluster quivers. We call $Q$ a {\bf tree cluster quiver} if it is a simple cluster quiver and $\bar Q$ is a tree.

Throughout this paper we mainly consider valued cluster quivers and we use the notation $[x]_+ = \mathrm{max}\{x, 0\}$. Let $M = (m_{ij})_{l\times n}$ be a real matrix, then $[M]_+ = {([m_{ij}]_+})_{l\times n}$ is the non-negative matrix defined component-wisely.

\begin{Definition}Let $(Q,v)$ be a valued cluster quiver with vertices set $Q_0 = \{1, 2, \dots, n\}$.
\begin{enumerate}
\item[(1)] The {\bf exchange matrix} $B(Q)=(b_{ij})_{n\times n}$ of $Q$ is the integer matrix defined by the following rule: For $1\leqslant i, j \leqslant n$, $b_{ij} = v_{ij}$ if there is an arrow $\alpha: i\rightarrow j$, $b_{ij} = -v_{ij}$ if there is an arrow $\beta: j\rightarrow i$ and otherwise $b_{ij} = 0$.
\item[(2)] The matrix $A(Q) = [B(Q)]_+ = ([b_{ij}]_+)_{n\times n}$ is called the {\bf adjacency matrix} of $Q$, and the matrix $C(Q) = [B(Q)]_+ + [-B(Q)]_+$ is called the adjacency matrix of the underlying valued graph $\bar Q$ of $Q$.
\end{enumerate}
\end{Definition}
A square matrix $M$ is {\bf symmetrizable} ({\bf skew-symmetrizable}, resp.) if there exists a diagonal square integer matrix $D$ with positive diagonal entries such that $DM$ is symmetric(skew-symmetric, resp.). Note that exchange matrices of valued cluster quivers are integer skew-symmetrizable matrices. Let $B=(b_{ij})_{n\times n}$ be an integer skew-symmetrizable matrix, we can define a valued cluster quiver $(Q(B),v)$ whose vertices set is $\{1,2,\dots,n\}$ as follows. There is an arrow $\alpha : i\rightarrow j$ in $Q(B)_1$ whenever $b_{ij}>0$ and $v(\alpha) = (|b_{ij}|,|b_{ji}|)$. It is clear that there is a bijective correspondence between integer skew-symmetrizable matrices and valued cluster quivers. In particular, skew-symmetric matrices are skew-symmetrizable. In this case, we can use cluster quivers instead of valued cluster quivers to express integer skew-symmetric matrices. Indeed, if $B=(b_{ij})_{n\times n}$ is an integer skew-symmetric matrix, we can construct a cluster quiver $Q$ such that $Q_0 = \{1,2,\dots,n\}$ and there shall be $b_{ij}$ arrows from $i$ to $j$ whenever $b_{ij}>0$ for any $i,j \in Q_0$. Then there is a bijective correspondence between integer skew-symmetric matrices and cluster quivers. Cluster quivers can be considered as a special case of valued cluster quivers if for any cluster quiver $Q$, we regard the multiplicity $v_{\alpha}$ of each arrow $\alpha$ as its value, that is, let $v(\alpha)_1 = v(\alpha)_2 = v_{\alpha}$ for each arrow $\alpha$. We will emphatically discuss cluster quivers in Section 4.

\begin{Remark}
When we consider skew-symmetric matrices and its corresponding cluster quivers, the adjacency matrix defined above is the same as the definition of the adjacency matrices of (oriented) graphs in \cite{GR,BH}.
\end{Remark}
Since a full valued subquiver of a valued cluster quiver is also a valued cluster quiver, the following relation between adjacency (exchange,resp.) matrices of a valued cluster quiver and its full valued subquivers is obvious.
\begin{Lemma}
Let $Q$ be a valued cluster quiver, then there is a bijection between principal submatrices of $A(Q) (B(Q), \mathrm{resp.})$ and full valued subquivers of $Q$. More precisely, each principal submatrix of $A(Q) (B(Q), \mathrm{resp.})$ is just the adjacency (exchange, resp.) matrix of its corresponding full valued subquiver in $Q$.
\end{Lemma}

Clearly, for a valued cluster quiver $Q$, $B(Q)$ is an integer skew-symmetrizable matrices and $C(Q)$ is an integer symmetrizable matrix with respect to the same positive definite diagonal matrix. The spectrum of $A(Q)$ ($B(Q)$, resp.) is called the {\bf adjacency} ({\bf exchange}, resp.) {\bf spectrum} of $Q$, and the characteristic polynomial of $A(Q)$ ($B(Q)$, resp.) is called the {\bf adjacency} ({\bf exchange}, resp.) {\bf polynomial} of $Q$. We usually denote the exchange spectrum of $Q$ by \[\mathrm{Spec}(B(Q)) = \begin{bmatrix} \lambda_1&\lambda_2&\dots&\lambda_m\\n_1&n_2&\dots&n_m \end{bmatrix}, \]where $\lambda_1, \lambda_2, \dots, \lambda_m$ are all distinct eigenvalues of the matrix $B(Q)$ such that $|\lambda_1| < |\lambda_2| < \dots < |\lambda_m|$, and $n_1$, $n_2$, $\dots$, $n_m$ are the corresponding multiplicities of them. And $|\lambda_m|$ is called the exchange spectrum radius of the valued cluster quiver $Q$ and denoted by $\mathrm{Radi}(Q)=|\lambda_m|$.

Let $Q$ be a valued cluster quiver and $Q'_1$, $Q'_2$, $\dots$, and $Q'_s$ be its connected components. Suppose that $B$, $B_1$, $B_2$, $\dots$, and $B_s$ are exchange matrices of $Q$, $Q'_1$, $Q'_2$, $\dots$, and $Q'_s$, respectively. Then it is clear that there exists a permutation matrix $P$ such that \[ PBP^T = \begin{bmatrix} B_1&&&\\&B_2&&\\&&\ddots&\\&&&B_s \end{bmatrix}\]Thus it is easy to see that \[\mathrm{Spec}(B) = \bigcup_{i=1}^s\mathrm{Spec}(B_i), \]\[\mathrm{Radi}(Q) = \mathrm{max}\{ \mathrm{Radi}(Q'_k),k=1,2,\dots,m\}.\] So in general, we can assume that $Q$ is connected.

For an integer skew-symmetrizable matrix $B = (b_{ij})_{n\times n}$ and any $k\in [1,n]$, let $\mu_k(B) = (b'_{ij})_{n\times n}$ be obtained by mutating $B$ at $k$, then $\mu_k(B)$ is defined by the following formula\[b'_{ij}=\begin{cases} -b_{ij},&\text{if $i=k$ or $j=k$;}\\b_{ij} + \mathrm{sgn}(b_{ik})\mathrm{max}\{b_{ik}b_{kj},0\},&\text{otherwise.}\end{cases}\]

Note that $\mu_k(B)$ is still an integer skew-symmetrizable matrix. The corresponding mutation of valued cluster quivers can be defined as follows.
\begin{Definition}
Let $(Q,v)$ be a valued cluster quiver with vertices set $Q_0 = \{1,2,\dots,n\}$ and $k\in Q_0$ be a fixed vertex. The mutation $(Q',v')=\mu_k(Q,v)$ of $(Q,v)$ at $k$ is defined as follows:
\begin{enumerate}
\item[(1)] For every 2-paths $i\xrightarrow[\alpha]{(v_{ik},v_{ki})} k\xrightarrow[\beta]{(v_{kj},v_{jk})} j$,
\begin{enumerate}
\item[(i)] if there exists an arrow $j\xrightarrow[\gamma]{(v_{ji},v_{ij})} i$, keep this arrow and $v'(\gamma)=(v_{ji}-v_{jk}v_{ki},v_{ij}-v_{ik}v_{kj})$ if $v_{ij}>v_{ik}v_{kj}$; delete this arrow if $v_{ij}=v_{ik}v_{kj}$; delete this arrow, then add a new arrow $\gamma': i\rightarrow j$ and $v'(\gamma')=(v_{ik}v_{kj}-v_{ij},v_{jk}v_{ki}-v_{ji})$ if $v_{ij}<v_{ik}v_{kj}$;
\item[(ii)] if there exists an arrow $i\xrightarrow[\gamma]{(v_{ij},v_{ji})} j$, keep this arrow and $v'(\gamma)=(v_{ik}v_{kj}+v_{ij},v_{jk}v_{ki}+v_{ji})$;
\item[(iii)] if there are not any arrows between $i$ and $j$, just add an arrow $\epsilon: i\rightarrow j$ and $v'(\epsilon)=(v_{ik}v_{kj},v_{jk}v_{ki})$.
\end{enumerate}
\item[(2)] Reverse all arrows incident with $k$, and $v'(\alpha^{op})=(v(\alpha)_2,v(\alpha)_1)$ for any arrow $\alpha$ incident with $k$, where $\alpha^{op}$ is the opposite arrow of $\alpha$;
\item[(3)] Keep other arrows and values unchanged.
\end{enumerate}
\end{Definition}
 It is obvious that $v_{ij}=|b_{ij}|$ when $b_{ij}\neq 0$. Furthermore, if there is an arrow from $i$ to $j$, then $v_{ij}=b_{ij}$ and $v_{ji}=-b_{ji}$. Let $\mu_k(B) = (b'_{ij})_{n\times n}$, then $b'_{ij}=b_{ij}+\mathrm{sgn}(b_{ik})[b_{ik}b_{kj}]_+$ for $i,j\neq k$. Therefore $b'_{ij}\neq b_{ij}$ if and only if $b_{ik}>0,b_{kj}>0$ or $b_{ik}<0,b_{kj}<0$ whenever $i,j\neq k$. It can be seen that $$B(\mu_k(Q,v))=\mu_k(B(Q,v)).$$

 Also, for either matrices or valued cluster quivers, the {\bf mutation map} $\mu_k$ is always an involution, that is, $\mu_k\mu_k(B) = B$ and $\mu_k\mu_k(Q,v) = (Q,v)$.

Mutations of an integer skew-symmetrizable matrix can be written in matrix form, see \cite{BFZ}. In particular, mutations of an integer skew-symmetric matrix is the same as a congruent transformation. Indeed,
let $B_{n\times n}$ be an integer skew-symmetric matrix and $W = (w_{ij})_{n\times n}$ be a matrix of the following form
\begin{equation}\label{pmatrix}
W= \begin{bmatrix}
  I_{k-1}&\xi&\mathbf{0}\\
 \mathbf{0}&-1&\mathbf{0}\\
 \mathbf{0}&\eta&I_{n-k}
  \end{bmatrix}  \end{equation}
satisfying that $\xi=([b_{1,k}]_+,[b_{2,k}]_+,\dots,[b_{k-1,k}]_+)^T$ and $\eta=([b_{k+1,k}]_+,[b_{k+2,k}]_+,\dots,[b_{n,k}]_+)^T$, where $I_m$ denotes the identity matrix of order $m$. It is easy to check that $det(W) = -1$ and $\mu_k(B) = WBW^T$.

Two integer skew-symmetrizable matrices (respectively, valued cluster quivers) are said to be {\bf mutation equivalent} if one can be obtained by a sequence of mutations of the other. It is easy to see that this defines a equivalence relation. The {\bf mutation class} of $Q$ consists of all valued cluster quivers mutation equivalent to $Q$ and is usually denoted by $Mut(Q)$. We use the notation $Q\sim Q'$($B\sim B'$, resp.) to denote that $Q$ and $Q'$ ($B$ and $B'$, resp.) are mutation equivalent.

Let $\mathbb{P}$ be a semifield which is an abelian multiplicative group endowed with an auxiliary addition $\oplus$ which is associative, commutative, and distributive with respect to the multiplication in $\mathbb{P}$. Let $F$ be a field which is isomorphic to the field of rational functions in $n$ indeterminates with the coefficients from the field of fractions of $\mathbb{ZP}$. Following \cite{FZ4}, a seed is a triple $\Sigma=(\mathbf{x},\mathbf{y},B)$ such that $B = (b_{ij})_{n\times n}$ is an integer skew-symmetrizable matrix, $\mathbf{y} = (y_1, \dots, y_n)$ is an n-tuple of elements of $\mathbb{P}$, and $\mathbf{x} = (x_1, \dots, x_n)$ is an n-tuple of a free generating set of $F$. For $k\in [1,n]$, $(\mathbf{x'},\mathbf{y'},B') = \mu_k(\mathbf{x},\mathbf{y},B)$ is obtained by the following rules:
\begin{enumerate}
\item[(1)] $\mathbf{x'} = (x'_1, \dots, x'_n)$ is given by $x'_kx_k = \frac{y_k\prod x_i^{[b_{ik}]_+} + \prod x_i^{[-b_{ik}]_+}}{y_k \oplus 1}$ and $x'_i = x_i$ for $i \neq k$;
\item[(2)] $\mathbf{y'} = (y'_1, \dots, y'_n)$ is given by $y'_i = y_k^{-1}$ for $i=k$; and otherwise $y'_i = y_iy_k^{[b_{ki}]_+}(y_k \oplus 1)^{-b_{ki}}$;
\item[(3)] $B' = \mu_k(B)$.
\end{enumerate}
Note that we can also use $(\mathbf{x},\mathbf{y},Q(B))$ instead of $(\mathbf{x},\mathbf{y},B)$. For every seed $(\tilde{\mathbf{x}},\tilde{\mathbf{y}},\tilde{B})$ obtained from the seed $\Sigma=(\mathbf{x},\mathbf{y},B)$ by a sequence of mutations, we call $\tilde{\mathbf{x}}$ a {\bf cluster} and its elements are called {\bf cluster variables}. The {\bf (rooted) cluster algebra} $\mathcal{A}(\Sigma)$ of rank $n$ associated to a seed $\Sigma=(\mathbf{x},\mathbf{y},B)$ is the $\mathbb{ZP}$-subalgebra of $F$ generated by all cluster variables.

\subsection{ A characterization of acyclic valued cluster quivers}.

 A valued cluster quiver $Q$ is called {\bf acyclic}, if it has no $k$-cycles in $Q$ for any $k \geqslant 1$. The fact whether a valued cluster quiver is acyclic will be influential for the corresponding cluster algebra. Indeed, some important conjectures were proved to be true in the case cluster algebras have acyclic valued cluster quivers; otherwise, however, they would face great difficult for affirmation. In this section, we give an criterion of acyclicity. The following lemma is easy to see:
\begin{Lemma}
Let $Q$ be a valued cluster quiver and $A=A(Q)$ be its adjacency matrix. Suppose $\pi$ is a permutation of $\{1, 2, \dots, n\}$, and $A'=(a'_{ij})$, where $a'_{ij} = a_{\pi(i)\pi(j)}$. If $P$ is the corresponding permutation matrix of $\pi$, then $PAP^T=A'$.
In particular, $det(A)=det(A')$.
\end{Lemma}

\begin{Proposition}\label{acyclic}
Let $Q$ be a valued cluster quiver. Then the following statements are equivalent:
\begin{enumerate}
\item[(i)] $Q$ is acyclic.
\item[(ii)] The principal minors of $A(Q)$ are zeros.
\item[(iii)] The eigenvalues of $A(Q)$ are zeros.
\end{enumerate}
\end{Proposition}
\begin{proof} (i)$\Rightarrow$ (ii): Since $Q$ is a finite acyclic quiver, there exists a bijection between $Q_0$ and $\{1, 2, \dots, n\}$ such that if we have an arrow $j\rightarrow i$, then $j< i$. Hence by the Lemma 2.5, there exists a permutation matrix $P$ such that $PA(Q)P^T = A'(Q)$, where $A'(Q)$ is a strictly upper triangular matrix. The principal minors of $A'(Q)$ are zeros, so are the principal minors of $A(Q)$.

(ii) $\Rightarrow$ (i): Assume that $Q$ is not acyclic, then there exists at least one cycle in $Q$. Let $Q'$ be a full valued subquiver of $Q$ such that $Q'$ has a cycle. We may, without loss of generality, assume $Q'$ to be minimal with this property. Then $Q'$ must be a chordless $k$-cycle. Suppose $Q'_0=\{j_1,j_2,\dots,j_k\}$ with $k\geqslant 3$, and there only exist arrows from $j_s$ to $j_{s+1}$ $(1\leqslant s \leqslant k-1)$ and from $j_k$ to $j_1$. It follows from the Lemma 2.5 that the adjacency matrix $A(Q')$ of $Q'$ is similar to the matrix \[\begin{bmatrix}
  0&c_1&0&\dots&0&0\\
 0&0&c_2&\dots&0&0\\
 \vdots&\vdots&\vdots& &\vdots&\vdots\\
 0&0&0&\dots&0&c_{k-1}\\
 c_k&0&0&\dots&0&0
  \end{bmatrix},\]where $c_i$ equals to the first element $v(\alpha_i)_1$ of $v(\alpha_i)$, where $\alpha_i$ is the arrow from $j_i$ to $j_{i+1}$ for $1\leqslant i \leqslant k-1$, and $c_k$ equals the first element $v(\alpha_k)_1$ of $v(\alpha_k)$, where $\alpha_k$ is the arrow from $j_k$ to $j_1$. Hence $det(A(Q'))=(-1)^{1+k}c_1\dots c_k\neq 0$. Then the principal minor of $A(Q)$ indexed by $\{j_1, j_2, \dots ,j_k \}$ is not zero, which is a contradiction.

(ii) $\Rightarrow$ (iii): Suppose that $f(\lambda) = |\lambda I_n - A(Q)| = \lambda^n + a_1\lambda^{n-1} + \dots + a_{n-1}\lambda + a_n.$ Since $(-1)^ sa_s$ is the sum of all principal minors of order $s$, we get $f(\lambda) = \lambda^n$, then (iii) holds.

(iii) $\Rightarrow$ (ii): We will prove any principal minor of the matrix $A(Q)$ of order $k$ is zero by induction on $k$. In the cases of $k= 1$ and $k= 2$, the conclusion follows from the definition of valued cluster quivers. Now we assume that this conclusion holds for any $m$, where $m\leqslant k-1 \leqslant n$. We consider the case of $m=k$. Because the principal minors of $A(Q)$ which has $l$ rows and $l$ columns are zeros for each $l\in [1,k-1]$, $Q$ has not any $l$-cycles for any $l\in [1,k-1]$ and so are all of its full valued subquivers of order $k$. Let $Q'$ be any full valued subquiver of order $k$. If the full valued subquiver $Q'$ is acyclic, then the corresponding principal minor is zero. Otherwise $Q'$ has a cycle and hence $Q'$ must be a chordless $k$-cycle. Similar to the proof in (ii)$\Rightarrow$(iii), we deduce that the corresponding principal minor has the sign $(-1)^{k+1}$. Then all of the nonzero principal minors of order $k$ share the same sign. Because the eigenvalues of $A(Q)$ are zeros, the sum of all of principal minors of order $k$ is zero. Now it is obvious that any principal minor of the matrix $A(Q)$ of order $k$ is zero. We finish the proof.\end{proof}

As the converse-negative result  of Proposition \ref{acyclic}, we have:
\begin{Corollary}
Assume that $A_{i_1i_2\dots i_k}$ is a principal submatrix of $A(Q)$ indexed by $i_1, i_2, \dots, i_k$. If $detA_{i_1i_2\dots i_k} \neq 0$, then there exists a full subquiver $Q'$ of $Q$ which is a cycle such that $Q'_0 \subseteq \{i_1 , i_2 , \dots , i_k\}$.
\end{Corollary}

\section{Spectra of exchange matrices}

In this section, we discuss firstly the relations between exchange spectrum of a valued cluster quiver and adjacency spectrum of its underlying valued graph, and secondly the relations between exchange spectrum of a valued cluster quiver and that of its full valued subquivers.

Let $Q$ be a valued cluster quiver, we now turn on the properties of the exchange matrix $B(Q)$. Since $B(Q)$ is skew-symmetrizable, there is a diagonal matrix $D$ with positive diagonal entries such that $DB(Q)$ is skew-symmetric. It is easy to check that $D^{\frac{1}{2}}B(Q)D^{-\frac{1}{2}}$ is real, skew-symmetric and similar to $B(Q)$. We will use this property frequently and some well-known properties for real skew-symmetric matrices are given as follows.
\begin{Lemma}
Let $B$ be a real skew-symmetric matrix of order $n$, then following
assertions hold:
\begin{enumerate}
\item[(i)] $det(B)\geq 0$. Moreover, if $n$ is odd, then $det(B)=0$.
\item[(ii)] The eigenvalues of $B$ appear in complex conjugate pairs, and any eigenvalue of $B$ is either an imaginary number or zero.
\item[(iii)] The sum of all the eigenvalues of $B$ is zero, and equals the sum of all the image parts of the eigenvalues of $B$.
\item [(iv)] There exists an orthogonal matrix $P$ such that \[PBP^T=\begin{bmatrix}
  0&&&&&\\
  &\ddots&&&&\\
  &&0&&&\\
  &&&B_1&&\\
  &&&&\ddots&\\
  &&&&&B_s
  \end{bmatrix}, \]where $B_k=\begin{bmatrix}
  0&b_k\\
  -b_k&0
  \end{bmatrix}$,$b_\in \mathbb{R}$, $1\leqslant k \leqslant s$, and $P^T$ is the transpose of $P$.
\item [(v)] There exists an unitary matrix $U$ such that $UBU^*$ is a diagonal matrix, where $U^*$ is the conjugate transpose of $U$.
\end{enumerate}
\end{Lemma}

\subsection{Relations between exchange spectrum and adjacency spectrum }.

The largest eigenvalue of the adjacency matrix of an unoriented graph has been studied well and this investigation is still active. It is also interesting to consider the exchange spectrum radius of a valued cluster quiver. Here we introduce some notations for the sake of the following proofs. Let $M = (m_{ij})_{l\times s}$ and $N = (n_{ij})_{l\times s}$ be two real matrices of the same size, $M>N(M\geqslant N)$ means $m_{ij}>n_{ij}(m_{ij}\geqslant n_{ij})$ for all $1\leqslant i\leqslant l,1\leqslant j\leqslant s$. For a real matrix $M = (m_{ij})_{m\times n}$, if $m_{ij}> 0(\geqslant 0)$ for any $i\in [1,m],j \in [1,n]$, $M$ is said to be positive(nonnegative) and is denoted by $M > 0(M \geqslant 0)$. And ${\rm i} = \sqrt{-1}$.
\begin{Lemma}
[\cite{BR}, Theorem 3.6.2]The eigenvalues of the complex matrix $M=(m_{ij})_{n\times n}$ of order $n$ lie in the region of the complex plane determined by the union of the $n$-closed discs \[T_i = \{x | |x - m_{ii}|\leq t_i\},\quad i=1,2,\dots,n.\]where $t_i=\sum_{j=1,j\neq i}^n |m_{ij}|, i=1,2,\dots,n.$
\end{Lemma}
\begin{Proposition} Suppose that $Q$ is a valued cluster quiver with $n = |Q_0|$ and the connected components $Q'_1, Q'_2, \dots, Q'_s$.
Let $B=B(Q)= (b_{ij})_{n\times n}$ be the exchange matrix of $Q$  and $C=C(Q)= (c_{ij})_{n\times n}$ be the adjacency matrix of the underlying valued graph $\bar{Q}$.  Let $h_i = \sum_{j=1}^n|b_{ij}| = \sum_{j=1}^nc_{ij}$ be the degree of the vertex $i\in Q_0$ and $h = \max\{h_1 , h_2 , \dots , h_n\}$. Then for the exchange spectrum radius $\lambda$ of $Q$ and the adjacency spectrum radius $\mu$ of the valued graph $\bar{Q}$, the following is satisfied that \[0\leqslant \lambda \leqslant \mu \leqslant h = r, \] for  $r_p = \max_{i\in (Q'_p)_0}\sum_{j\in (Q'_p)_0}|b_{ij}|, 1\leqslant p\leqslant s,$ and $r = \max\{r_1 , r_2 , \dots , r_s\}$.
\end{Proposition}
\begin{proof} $0\leqslant \lambda$ and $h = r$ are obvious by definitions.

$\mu \leqslant h$ is just a corollary of Lemma 3.2 by letting $M=C(Q)$ in Lemma 3.2.

Let us show that the inequation $\lambda\leqslant \mu$ holds. Since $B$ is a skew-symmetrizable matrix, there exists a diagonal matrix $D = diag(d_1, d_2, \dots, d_n)$ with $d_i>0$ for $1\leqslant i \leqslant n$ such that $DB$ is skew-symmetric. Then it is clear that $B' = (b'_{ij}) = D^{\frac{1}{2}}BD^{-\frac{1}{2}}$ is skew-symmetric, $C' = (c'_{ij}) = D^{\frac{1}{2}}CD^{-\frac{1}{2}}$ is symmetric and $c'_{ij} = |b'_{ij}|$ for any $i, j \in [1,n]$. Let $x = (x_1, x_2, \dots, x_n)^T$ be an eigenvector of $B'$ corresponding to $\lambda {\rm i}$, say, $B'x = (\lambda {\rm i})x$.Thus for any $j \in [1,n]$, we have
\[ \lambda |x_j| = |\lambda {\rm i}x_j| = |\sum_{s=1}^nb'_{js}x_s|\leqslant \sum_{s=1}^n|b'_{js}||x_s| = \sum_{s=1}^nc'_{js}|x_s|. \]Let $y = (|x_1|, |x_2|, \dots, |x_n|)^T$, we  get $ C'y \geqslant \lambda y \geqslant 0$, and $ y\geqslant0,y\neq0.$ Therefore, $\lambda y^Ty \leqslant y^TC'y.$

Now let $z$ be an eigenvector of $C'$ corresponding to $\mu$. By the Rayleith theorem, we have\[\lambda \leqslant \frac{y^TC'y}{y^Ty} \leqslant \frac{z^TC'z}{z^Tz} = \mu. \]
\end{proof}
\begin{Proposition}
With the notations above, if $\lambda = h$, then there exists a full valued subquiver $Q'$ of $Q$ which is also a connected component of $Q$ such that $h_i=\sum_{j\in (Q')_0}|b_{ij}|=h$, for each vertex $i\in (Q')_0$.
\end{Proposition}
\begin{proof} Assume that $h{\rm i}$ is an eigenvalue of $B(Q)$. If the set of arrows $Q_1 = \varnothing$, it is obvious. Now we assume that $Q_1 \neq \varnothing$ so that $h \neq 0$. Let $Q'_1, Q'_2, \dots , Q'_s$ be all connected components of the valued cluster quiver $Q$ and $B_i=B(Q'_i)$ be the exchange matrix of $Q'_i$ for $i=1,2,\dots,s$. Then we have \[\mathrm{Spec}(B(Q)) = \bigcup_{i=1}^s\mathrm{Spec}(B_i). \]We may assume that $h{\rm i}$ is an eigenvalue of $B_k$ for some $k \in [1, s]$. Without loss of generality, we assume that $(Q'_k)_0 = \{1,2,\dots ,m\},m \geqslant 2$, $B_k = (b_{ij})_{m\times m}$, and $C_k = C(Q'_k) = (c_{ij})_{m\times m}$. Since $h$ is the exchange spectrum radius of $Q'_k$, it follows from Proposition 3.4 that $h$ is the largest eigenvalue of $C_k$. It is also clear that \[C_k(1, 1, \dots, 1)^T \leqslant h(1, 1, \dots, 1)^T.\]
Since $\bar{Q'_k}$ is connected, the symmetrizable matrix $C_k$ is irreducible and $C_k \geqslant 0$. It follows from the Perron-Frobenius theorem that $C_k(1, 1, \dots, 1)^T = h(1, 1, \dots, 1)^T$. Now it is easy to see that the connected component $Q'_k$ is what we need.
\end{proof}
From Proposition 3.3, for any valued cluster quiver, we know its exchange spectrum radius is not more than the adjacency spectrum radius of its underlying valued graph. In particular, when its underlying graph is a tree, we can say more.
\begin{Proposition}
Let $Q$ be a valued cluster quiver with $Q_0=n$ and its underlying graph $\bar Q$ be a tree. Assume that $f(x)$ and $g(x)$ are the exchange polynomial of $Q$ and the adjacency polynomial of the underlying valued graph $\bar Q$ respectively, that is, $f(x)=|xI_n - B(Q)|,\;g(x)=|xI_n - C(Q)|.$ Then, for  $\lambda \in \mathbb R$,  $f(\lambda {\rm i})= 0$ if and only if $g(\lambda)=0$. Moreover, it holds that \[\mathrm{Spec}(B(Q))= \begin{bmatrix} \lambda_0{\rm i}&\lambda_1{\rm i}&\dots&\lambda_p{\rm i}\\n_0&n_1&\dots&n_p \end{bmatrix} \;\;\;\;\;\text{if and only if}\;\;\;\;\; \mathrm{Spec}(C(Q))= \begin{bmatrix} \lambda_0&\lambda_1&\dots&\lambda_p\\n_0&n_1&\dots&n_p \end{bmatrix} \]where $0\leqslant \lambda_1 < \lambda_2 < \dots <\lambda_p$,and $n_0 + n_1 + \dots +n_p = n$.
\end{Proposition}
\begin{proof} Note that $B=B(Q)=(b_{ij})_{n\times n}$ is skew-symmetrizable and $C=C(Q)=(c_{ij})_{n\times n}$ is symmetrizable with the same diagonal matrix $D$, and it is easy to see that $|b_{ij}|= c_{ij}$. We have $D^{\frac{1}{2}}BD^{-\frac{1}{2}}$ is skew-symmetric and $D^{\frac{1}{2}}CD^{-\frac{1}{2}}$ is symmetric.

At first, we prove that $det(B)=(-1)^{\frac{n}{2}}det(C).$

We have $det(B)=\sum_{\pi\in S_n}\mathrm{sgn}(\pi)b_{1\pi1}b_{2\pi2}\dots b_{n\pi n},\;  det(C)=\sum_{\pi\in S_n}\mathrm{sgn}(\pi)c_{1\pi1}c_{2\pi2}\dots c_{n\pi n}$, where $S_n$ means the permutation group of $\{1,2,\dots,n\}$. Because of the definition of $B$, $b_{ij} \neq 0$ if and only if the two vertices $i$ and $j$ are adjacent in $\bar{Q}$. In particular, if $i=\pi i$, then $b_{i\pi i}=0$. If $\pi$ is not the identity, then $\pi$
can be uniquely expressed to be a product of disjoint cycles of length at least two. Let the cycle $(spr\dots t)$ be a factor of length more than two of $\pi$, then it corresponds to the factor $b_{sp}b_{pr}\dots b_{ts}$ of the term $\mathrm{sgn}(\pi)b_{1\pi 1}\dots b_{n\pi_n}$. And $b_{sp}b_{pr}\dots b_{ts} \neq 0$ if and only if the pairs $\{s,p\}$, $\{p,r\}$, $\dots$, $\{t,s\}$ are adjacent pairs in $\bar{Q}$. In this case, the induced subgraph of $\bar Q$ determinded by $\{s,p,r,\dots,t\}$ admits a cycle of length more than two. But we know that there are no $k$-cycles for $k\geqslant 3$ in $\bar Q$. Thus if the term $\mathrm{sgn}(\pi)b_{1\pi1}b_{2\pi2}\dots b_{n\pi n}$ does not vanish, $\pi$ must be a product of disjoint cycles of length two. The same statements hold for $C$.

 If $n= |Q_0|$ is odd, then $det(B)=0$ follows from Lemma 3.1(1) and the fact that $D^{\frac{1}{2}}BD^{-\frac{1}{2}}$ is skew-symmetric. For any $\pi \in S_n$, if the term $\mathrm{sgn}(\pi)c_{1\pi1}c_{2\pi2}\dots c_{n\pi n}\neq 0$, then it implies $\pi i\neq i$ for any $i\in[1,n]$ and $\pi$ is a product of disjoint cycles of length two for $\bar{Q}$ is a tree, which means it is impossible for $n= |Q_0|$ to be odd. Thus, $\mathrm{sgn}(\pi)c_{1\pi1}c_{2\pi2}\dots c_{n\pi n}= 0$ for any $\pi$. Hence $det(B)=0=det(C)$.

  If $n= |Q_0|$ is even, it is easy to see that for $\pi\in S_n$, $b_{1\pi1}b_{2\pi2}\dots b_{n\pi n}\neq 0$ if and only if $c_{1\pi1}c_{2\pi2}\dots c_{n\pi n}\neq0$ and in this case, $\pi$ is a product of disjoint cycles of length two. Note that $b_{ij}b_{ji}=-c_{ij}c_{ji}$, we have $\mathrm{sgn}(\pi)b_{1\pi1}b_{2\pi2}\dots b_{n\pi n} = (-1)^{\frac{n}{2}}\mathrm{sgn}(\pi)c_{1\pi1}c_{2\pi2}\dots c_{n\pi n}.$ for any $\pi\in S_n$. Thus it follows $det(B)=(-1)^{\frac{n}{2}} det(C)$.

Because the underlying graphs of full valued subquivers of $Q$ do not have $l$-cycles for $l\geqslant 3$ either, then for any full valued subquiver $Q'$ of $Q$ of order $r$, we have $det(B(Q')) = (-1)^{\frac{r}{2}}det(C(Q'))$. By the relations between coefficients of characteristic polynomials and principal minors, we have the following statements.

When $n$ is even, let $n= 2m$. Note that all principal minors of $B$ of odd orders are zeros, we may assume that \[f(\lambda)= \lambda^{2m} + v_2\lambda^{2m-2} + v_4\lambda^{2m-4} + \dots + v_{2m-2}\lambda^2 + v_{2m}, \]where $(-1)^kv_k$ is the sum of all of principal minors of $B$ of order $k$. Then \[g(\lambda)= \lambda^{2m} + (-1)v_2\lambda^{2m-2} + (-1)^2v_4\lambda^{2m-4} + \dots + (-1)^{m-1}v_{2m-2}\lambda^2 + (-1)^mv_{2m}. \]It is easy to see that $f(\lambda {\rm i})= (-1)^mg(\lambda)$.

When $n$ is odd, let $n= 2m+1$. Similarly, we may assume that \[f(\lambda)= \lambda^{2m+1} + v_2\lambda^{2m-1} + v_4\lambda^{2m-3} + \dots + v_{2m-2}\lambda^3 + v_{2m}\lambda, \] and  then, $g(\lambda)= \lambda^{2m+1} + (-1)v_2\lambda^{2m-1} + (-1)^2v_4\lambda^{2m-3} + \dots + (-1)^{m-1}v_{2m-2}\lambda^3 + (-1)^mv_{2m}\lambda.$ It is clear that $f(\lambda {\rm i})= {\rm i}(-1)^mg(\lambda)$.

Thus in all cases, we have that $f(\lambda {\rm i})= 0$ if and only if $g(\lambda)=0$.

Moreover, let $f(\lambda)= (\lambda^2 + q_1)(\lambda^2 + q_2)\dots(\lambda^2 + q_s)\lambda^{n-2s}$ for $0< q_1 \leqslant q_2 \leqslant \dots \leqslant q_s $.

When $n= 2m$, we have \[\begin{split}(-1)^mg(\lambda) & = f(\lambda {\rm i}){} = (\lambda^2 - q_1)(\lambda^2 - q_2)\dots(\lambda^2 - q_s)\lambda^{n-2s}(-1)^s(\rm i)^{n-2s}{}\\
&= (\lambda^2 - q_1)(\lambda^2 - q_2)\dots(\lambda^2 - q_s)\lambda^{n-2s}(-1)^m. \end{split}\] Thus, $g(\lambda)= (\lambda^2 - q_1)(\lambda^2 - q_2)\dots(\lambda^2 - q_s)\lambda^{n-2s}$.

When $n= 2m+1$, we have \[\begin{split}{\rm i}(-1)^mg(\lambda) & = f(\lambda {\rm i}) = (\lambda^2 - q_1)(\lambda^2 - q_2)\dots(\lambda^2 - q_s)\lambda^{n-2s}(-1)^s(\rm i)^{n-2s} \\
&= (\lambda^2 - q_1)(\lambda^2 - q_2)\dots(\lambda^2 - q_s)\lambda^{n-2s}(-1)^m{\rm i}. \end{split}\] Thus,  $g(\lambda)= (\lambda^2 - q_1)(\lambda^2 - q_2)\dots(\lambda^2 - q_s)\lambda^{n-2s}.$

Hence if $\mathrm{Spec}(f)= \begin{bmatrix} \lambda_0{\rm i}&\lambda_1{\rm i}&\dots&\lambda_p{\rm i}\\n_0&n_1&\dots&n_p \end{bmatrix}$, then $\mathrm{Spec}(g)= \begin{bmatrix} \lambda_0&\lambda_1&\dots&\lambda_p\\n_0&n_1&\dots&n_p \end{bmatrix}$.

 The proof of its converse statement is similar.
\end{proof}
A valued cluster quiver $(Q',v')$ is said to be obtained by re-orienting an arrow $\alpha$ from a valued cluster quiver $(Q,v)$ if $Q'_1 = \{\alpha^{op}\}\bigcup Q_1\setminus \{\alpha\}$, $(v'(\alpha^{op})_1,v'(\alpha^{op})_2)=(v(\alpha)_2,v(\alpha)_1)$ and $v'(\beta)=v(\beta)$ for $\beta \in Q_1\setminus\{\alpha\}$, where $\alpha^{op}$ is the opposite arrow of $\alpha$. Re-orientations of a valued cluster quiver by re-orienting a set of arrows are defined step by step. Then we have the following corollary.

\begin{Corollary}
Let $T$ be a full valued subquiver of a connected valued cluster quiver $Q$ such that:
\begin{enumerate}
\item[(i)] The underlying graph \={T} of $T$ is a tree.
\item[(ii)] There is only one vertex $x \in T_0$ connecting with the vertices in $Q_0\setminus T_0$.
\end{enumerate}
Then all re-orientations of the valued cluster quiver $Q$ by re-orienting $T$ and maintaining $T'$ unchanged share the same exchange polynomial, where $T'$ is a full valued subquiver of $Q$ determined by $Q_0\setminus T_0$.

In particular, if the underlying graph $\bar Q$ of $Q$ is a tree, then all re-orientations of $Q$ share the same exchange polynomial.
\end{Corollary}
\begin{proof} Without loss of generality, we may suppose that $T'_0 = \{1, 2, \dots, m\}$, $T_0 = \{m+1, m+2, \dots, m+n\}$ and $x = m+1$. The exchange matrices of $T'$ and $T\setminus \{x\}$ are assumed to be $X_{m\times m}$ and $Y_{(n-1)\times (n-1)}$, respectively. Then the exchange matrix of $Q$ will be the following form:\[\begin{bmatrix} X&-w^T&0\\\alpha&0&\beta\\0&-y^T&Y \end{bmatrix}, \] where $\alpha = (\alpha_1, \alpha_2, \dots, \alpha_m)$ , $\beta = (\beta_1, \beta_2, \dots, \beta_{n-1})$,$w = (w_1, w_2, \dots, w_m)$ , and $y = (y_1, y_2, \dots, y_{n-1})$. Assume that characteristic polynomials of $X$ and $Y$ are $X(\lambda)$ and $Y(\lambda)$, respectively. Then the exchange polynomial of $Q$ is

\[\begin{split} &\begin{vmatrix} \lambda I_m - X&w^T&0\\-\alpha&\lambda&-\beta\\0&y^T&\lambda I_{n-1}-Y\end{vmatrix}\\ = &(\alpha_1w_1X_1(\lambda) + \alpha_2w_2X_2(\lambda) + \dots + \alpha_mw_mX_m(\lambda))Y(\lambda)\\ & + \lambda X(\lambda)Y(\lambda) + [\beta_1y_1Y_1(\lambda) + \beta_2y_2Y_2(\lambda) + \dots + \beta_{n-1}y_{n-1}Y_{n-1}(\lambda)]X(\lambda), \end{split} \]
where $X_k(\lambda)$ is the determinant of the principal sbumatrix of the matrix $\lambda I_m - X$ obtained by deleting the $k$-th row and $k$-th column, and $Y_j(\lambda)$ is the determinant of the principal sbumatrix of $\lambda I_{n-1} - Y$ obtained by deleting the $j$-th row and $j$-th column for any $k \in [1,m], j \in [1,n-1]$. Re-orientations of $Q$ with $T'$ unchanged will keep $X(\lambda)$, $X_1(\lambda)$, $\dots$, $X_m(\lambda)$, $\alpha_1w_1$, $\dots$, $\alpha_mw_m$, $\beta_1y_1$, $\dots$, $\beta_{n-1}y_{n-1}$ unchanged. It only needs to show that $Y(\lambda)$, $Y_1(\lambda)$, $\dots$, $Y_{n-1}(\lambda)$ stay unchanged, and this follows immediately from Proposition 3.5.
\end{proof}

For any orientation of a tree, we may get a tree cluster quiver. Recall that an {\bf induced subgraph}(or say, a {\bf full subgraph}) of a graph is a subgraph obtained from the original graph by keeping an arbitrary subset of vertices together with all the edges that have both endpoints in this subset. We have the following results for (tree) cluster quivers on exchange spectrum radii.
\begin{Corollary}The following assertions hold:
\begin{enumerate}
\item[(1)] Let $Q$ be a tree cluster quiver, then
\begin{enumerate}
\item[(i)] The exchange spectrum radius of $Q$ is less than two if and only if the underlying graph of $Q$ is one of Dynkin diagrams (see Figure \ref{Dynkin}).
\begin{figure}
\centering
\includegraphics[height=4cm,width=10cm]{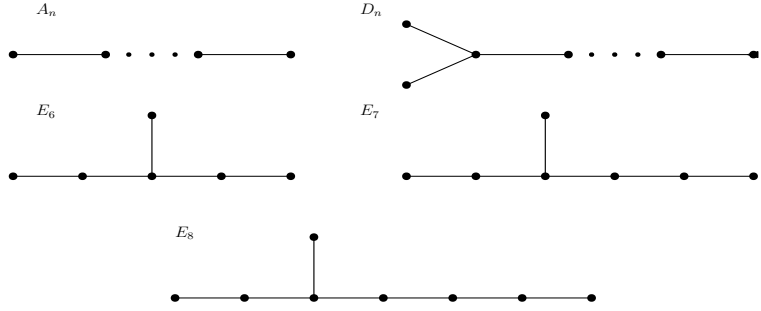}
\caption{Dynkin diagrams $A_n$, $D_n$, $E_6$, $E_7$, and $E_8$}
\label{Dynkin}
\end{figure}
\item[(ii)] The exchange spectrum radius of $Q$ is $2$ if and only if the underlying graph of $Q$ is one of the graphs $\hat{D_n}$ ($n\geqslant4$, with $n+1$ vertices in it), $\hat{E_6}$, $\hat{E_7}$, or $\hat{E_8}$ (see Figure \ref{ADynkin}).
\begin{figure}
\centering
\includegraphics[height=5cm,width=10cm]{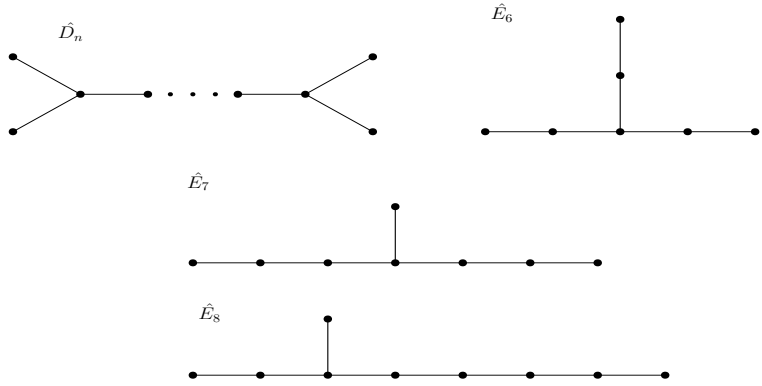}
\caption{Acyclic extended Dynkin diagrams $\hat{D_n}$, $\hat{E_6}$, $\hat{E_7}$, and $\hat{E_8}$}
\label{ADynkin}
\end{figure}
\item[(iii)] The exchange spectrum radius of $Q$ is more than two if and only if the underlying graph of $Q$ contains $\hat{D_n} (n\geqslant4)$, $\hat{E_6}$, $\hat{E_7}$, or $\hat{E_8}$ as a proper induced subgraph.
\end{enumerate}
\item[(2)] Let $Q$ be a cluster quiver with exchange spectrum radius more than two, then $\bar{Q}$ either contains $X_2$, $\hat{A_n}$ $(n\geqslant2)$, $\hat{D_n}$ $(n\geqslant4)$, $\hat{E_6}$, $\hat{E_7}$, or $\hat{E_8}$ as a proper induced subgraph, or contains $X_n(n\geqslant3)$ as an induced subgraph, where $\hat{A_n}$ is a simple chordless $(n+1)$-cycle, and $X_n$ is a graph with two vertices and $n$ edges.
\end{enumerate}
\end{Corollary}
\begin{proof} (1) It follows from the Proposition 3.5 that the exchange spectrum radius of a valued cluster quiver equals to the adjacency spectrum radius of its underlyling graph if its underlying graph is a tree. Then the conclusion follows from results in \cite{LS}(refer also to Theorem 3.1.3 in \cite{BH}).

(2) This assertion follows from (1), Proposition 3.3 and results in \cite{LS}.
\end{proof}
Dynkin diagrams have appeared in many branches of mathematics, for example, in the classification of finite type of cluster algebras, finite dimensional associated algebras, and Lie algebras. It is also interesting to see they can be associated with exchange matrices of valued cluster quivers.

\subsection{Exchange spectra of a valued cluster quiver and full valued subquivers}.

In this subsection, we make use of Cauchy's interlacing theorem for symmetric matrices to prove a similar result for skew-symmetrizable matrices, and we use this result to compare the exchange radii of valued cluster quivers and their full valued subquivers.
\begin{Theorem}
Let $Q$ be a valued cluster quiver with $|Q_0|=n$ and $Q'$ be a full valued subquiver order $n-1$ of $Q$. If the eigenvalues of $B(Q)$ are $\lambda_1{\rm i}, \lambda_2{\rm i}, \dots, \lambda_n{\rm i}$ with $\lambda_1 \geqslant \lambda_2 \geqslant \dots \geqslant \lambda_n$ and the eigenvalues of $B(Q')$ are $\gamma_2{\rm i}, \gamma_3{\rm i}, \dots , \gamma_n{\rm i}$ with $\gamma_2 \geqslant \gamma_3 \geqslant \dots \geqslant \gamma_n$, then $\lambda_1 \geqslant \gamma_2 \geqslant \lambda_2 \geqslant \gamma_3 \geqslant \lambda_3 \geqslant \dots \geqslant \gamma_n \geqslant \lambda_n$.
\end{Theorem}
\begin{proof} Since $B=B(Q)$ is a real skew-symmetrizable matrix and $B'=B(Q')$ is a principal submatrix of order $n-1$ of $B(Q)$, there exists a diagonal matrix $D$ with positive diagonal entries such that $DB$ is skew-symmetric. It is clear that $B_1 = D^{\frac{1}{2}}BD^{-\frac{1}{2}}$ and $B'_1 = D_1^{\frac{1}{2}}B'D_1^{-\frac{1}{2}}$ are skew-symmetric, where $D_1$ is the corresponding principal submatrix of $D$ such that $D_1B'$ is skew-symmetric. It is also obvious that there is a permutation matrix $P$ such that \[PB_1P^T= \begin{bmatrix} 0&\alpha^*\\-\alpha&B'_1 \end{bmatrix}, \] where $\alpha^*$ is the conjugate transpose of the vector $\alpha$. Since $B'_1$ is a real skew-symmetric matrix, there exists an unitary matrix $T$ such that $TB'_1T^*$ is a diagonal matrix, ~i.~e. \[TB'_1T^* = \begin{bmatrix} \gamma_2{\rm i}&&&\\&\gamma_3{\rm i}&&\\&&\ddots &\\&&&\gamma_n{\rm i} \end{bmatrix},\]where $T^*$ is the conjugate transpose of $T$. Now let  a matrix $H = \begin{bmatrix} 1&0\\0&T \end{bmatrix}$. Then we have \[ HPB_1P^TH^* = \begin{bmatrix} 1&0\\0&T \end{bmatrix}\begin{bmatrix}0&\alpha^*\\-\alpha&B'_1 \end{bmatrix}\begin{bmatrix} 1&0\\0&T^* \end{bmatrix} = \begin{bmatrix} 0&\alpha^*T^*\\-T\alpha&TB'_1T^* \end{bmatrix} = \begin{bmatrix} 0&\beta^*\\-\beta&TB'_1T^* \end{bmatrix},\]where $\beta = T\alpha = (\beta_2 , \beta_3 , \dots , \beta_n)^T$. Assume that the characteristic polynomial of $HPB_1P^TH^*$ is $f(\lambda)$, then
\[\begin{split} f(\lambda) = |\lambda I_n - HPB_1P^TH^*| &= \begin{vmatrix} \lambda &-\beta^* \\ \beta &\lambda I_{n-1} - TB'_1T^* \end{vmatrix} = \begin{vmatrix} \lambda &-\bar{\beta_2}& -\bar{\beta_3}& \dots & -\bar{\beta_n} \\ \beta_2&\lambda - \gamma_2{\rm i}&0&\dots&0\\ \beta_3&0&\lambda - \gamma_3{\rm i}&\dots&0\\ \vdots&\vdots&\vdots&\ddots&\vdots\\ \beta_n&0&0&\dots&\lambda - \gamma_n{\rm i} \end{vmatrix}. \end{split}\]
Expanding the above determinant along the first row, we get
\[  f(\lambda) = \lambda(\lambda - \gamma_2{\rm i})\dots(\lambda-\gamma_n{\rm i})+ \sum_{k=2}^{n}|\beta_k|^2(\lambda-\gamma_2{\rm i})\dots\widehat{(\lambda-\gamma_k{\rm i})}\dots(\lambda-\gamma_n{\rm i}),  \]
where $\widehat{(x-\gamma_k)}$ means deleting this term. Thus it follows that
 \[\begin{split} f(x{\rm i})=&(x{\rm i})(x{\rm i}-\gamma_2{\rm i})\dots (x{\rm i}-\gamma_n{\rm i})+ \sum_{k=2}^{n}|\beta_k|^2(x{\rm i}-\gamma_2{\rm i})\dots\widehat{(x{\rm i}-\gamma_k{\rm i})}\dots(x{\rm i}-\gamma_n{\rm i})\\ =& {\rm i}^nx(x-\gamma_2)\dots (x-\gamma_n)+ {\rm i}^{n-2}\sum_{k=2}^{n}|\beta_k|^2(x-\gamma_2)\dots\widehat{(x-\gamma_k)}\dots(x-\gamma_n)\\=&{\rm i}^ng(x), \end{split}\]
 where $g(x)= x(x-\gamma_2)\dots(x-\gamma_n)- \sum_{k=2}^{n}|\beta_k|^2(x-\gamma_2)\dots\widehat{(x-\gamma_k)}\dots(x-\gamma_n)$. It is easy to see that $g(x)$ is the characteristic polynomial of the real symmetric matrix $M$ defined as follows.\[ M = \begin{bmatrix} 0&|\beta_2|&\dots&|\beta_n|\\|\beta_2|&\gamma_2&\dots&0\\\vdots&\vdots&\ddots&\vdots\\|\beta_n|&0&\dots&\gamma_n \end{bmatrix}.\]Similar to the proof of Proposition 3.5, it is not difficult to see that the eigenvalues of the matrix $B$ are $\lambda_1{\rm i}$, $\lambda_2{\rm i}$, $\dots$ , $\lambda_n{\rm i}$ if and only if the eigenvalues of the matrix $M$ are $\lambda_1$, $\lambda_2$, $\dots$ , $\lambda_n$. According to the Cauchy's interlacing theorem for Hermitian matrices (see e.g. \cite{HJ}), our proof is finished. \end{proof}

\begin{Corollary}
Let $Q$ be a valued cluster quiver with $|Q_0| = n$ and $Q'$ be a full valued subquiver of order $m(<n)$ of $Q$. Suppose that the eigenvalues of $B(Q)$ are $\lambda_1{\rm i}, \lambda_2{\rm i}, \dots, \lambda_n{\rm i}$ with $\lambda_1 \geqslant \lambda_2 \geqslant \dots \geqslant \lambda_n$, and the eigenvalues of $B(Q')$ are $\gamma_1{\rm i}, \gamma_2{\rm i}, \dots,\gamma_m{\rm i}$ with $\gamma_1 \geqslant \gamma_2 \geqslant \dots \geqslant \gamma_m$. Then $\lambda_j \geqslant \gamma_j \geqslant \lambda_{n-m+j}$ for $j= 1, 2, \dots, m$.

In particular, the exchange spectrum radius of $Q$ is either larger than or equal to that of $Q'$.
\end{Corollary}
\begin{proof} Because the exchange matrix of the full valued subquiver $Q'$ is a principal submatrix of the exchange matrix of $Q$, the conclusion follows from iterated applications of Theorem 3.8.
\end{proof}

\section{Mutation invariant of spectrum of a cluster quiver}

In this section, we study mutation invariants of skew-symmetric matrices and cluster quivers under the meaning of spectrum.

As a special case of Definition 2.3, mutation of cluster quivers can be equivalently defined as follows.
\begin{Definition}[\cite{M}]
Let $Q$ be a cluster quiver and $k\in Q_0$ be a fixed vertex. The mutation $\mu_k(Q)$ of $Q$ at $k$ is defined as follows:
\begin{enumerate}
\item[(1)] For every 2-path $i\rightarrow k\rightarrow j$, add a new arrow $i\rightarrow j$;
\item[(2)] Reverse all arrows incident with $k$;
\item[(3)] Delete a maximal collection of 2-cycles from those created in $(1)$.
\end{enumerate}
\end{Definition}
Note that parallel arrows are considered as different arrows in the first step in Definition 4.1.
\subsection{Cospectral relationship of cluster quivers and seeds}.

Two cluster quivers $Q$ and $Q'$ are called {\bf cospectral} if they share the same exchange polynomial. Two mutation equivalent seeds $\Sigma=(\mathbf{x},\mathbf{y},Q)$ and $\Sigma'=(\mathbf{x'},\mathbf{y'},Q')$ are said to be {\bf cospectral} if $Q$ and $Q'$ are cospectral. In this case, we call clusters $\mathbf{x}$ and $\mathbf{x'}$ {\bf cospectral} and denoted by $\mathbf{x}\thicksim_c \mathbf{x'}$.

Firstly, we consider the condition for cluster quivers to be cospectral in a mutation class.
\begin{Lemma}
Let $Q$ be a cluster quiver and $B$ is its exchange matrix. If $k\in Q_0$ is a sink or source, then the quivers $Q$ and $\mu_k(Q)$ are cospectral.
\end{Lemma}
\begin{proof} If $k\in Q_0$ is a sink or source, there are no $2$-paths of the form $i\rightarrow k\rightarrow j$, $\mu_k(Q)$ is obtained from $Q$ just by reversing all arrows incident with $k$. Then the exchange matrix of $\mu_k(Q)$ is given by \[\mu_k(B) = J_kBJ_k,\]where $J_k$ is the diagonal matrix obtained from the identity matrix by replacing the $(k,k)$-entry by $-1$. Obviously, the quivers $Q$ and $\mu_k(Q)$ are cospectral.
\end{proof}
The following example shows that in general, the converse of Lemma 4.2 is not true.
\begin{Example}
Let us consider the cluster quiver in Figure \ref{Exa}(a):
\begin{figure}
\centering
\includegraphics[height=3cm,width=10cm]{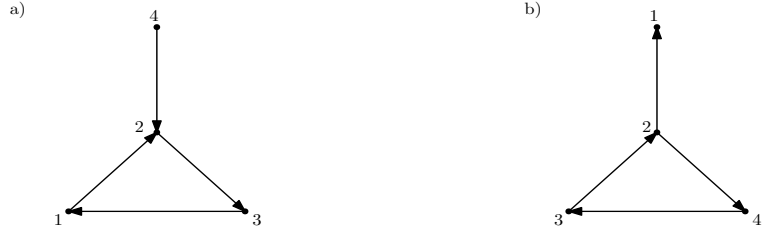}
\caption{An example}
\label{Exa}
\end{figure}
Mutating at vertex $2$, we get a quiver in Figure \ref{Exa}(b). It is obvious that $\mu_2(Q)$ and $Q$ have the same exchange polynomial, see Remark 4.12. However the vertex $2$ is neither a sink nor a source.
\end{Example}

If we consider cluster quivers without $3$-cycles, we have a better result. To prove the desired result, we need the following lemma.
\begin{Lemma}
Let $B(Q)=(b_{ij})_{n\times n}$ be the exchange matrix of a cluster quiver $Q$. The exchange polynomial of $Q$ is $f(\lambda)=|\lambda I_n -B(Q)|= \lambda^n + b_1\lambda^{n-1} + \dots +b_{n-1}\lambda + b_n$, then

(i)\; $b_{2k-1} = 0$, $1 \leq k \leq \frac{n+1}{2}, k \in \mathbb Z$; \;\;\;\;
(ii)\; $b_2 = \sum_{i<j}b_{ij}^2 = \sum_{j<i}b_{ij}^2$.
\\
In particular, if $Q$ is a simply-laced cluster quiver, then $b_2$ equals to the number of arrows in $Q$.
\end{Lemma}
\begin{proof} Since the coefficient $b_k$ of the characteristic polynomial equals to the sum of all principal minors of order $k$ multiplying by $(-1)^k$, by Lemma 3.1, the determinants of skew-symmetric matrices of odd orders are zeros, then (i) is true. The principal minor of order two must be of the form \[ \begin{vmatrix} 0&b_{ij}\\-b_{ij}&0 \end{vmatrix}, \]where $1 \leqslant i < j \leqslant n$. Then (ii) is also true. For a simply-laced cluster quiver, the nonzero principal minors of order two must be \[ \begin{vmatrix} 0&1\\-1&0 \end{vmatrix}\quad or \quad \begin{vmatrix} 0&-1\\1&0 \end{vmatrix}. \] And there is a bijection between the nonzero principal minors of order two and arrows in $Q_1$. Then the last statement follows.
\end{proof}
Let $Q$ be a cluster quiver and $B=B(Q)$. Given the matrix $W$ at $k$ satisfying the equality (\ref{pmatrix}), we have $\mu_k(B) = WBW^T$.
\begin{Proposition}
Let $Q$ be a cluster quiver without $3$-cycles and $Q_0=\{1,2,\dots,n\}$. Fix a vertex $k\in Q_0$, $B=B(Q)=(b_{ij})_{n\times n}$ and $\mu_k(B) = WBW^T$, where $W$ satisfies the equality (\ref{pmatrix}). The following statements are equivalent:
\begin{enumerate}
\item[(i)] $Q$ and $\mu_k(Q)$ are cospectral;
\item[(ii)] $k$ is either a sink or a source;
\item[(iii)] $ W= \begin{bmatrix}
  I_{k-1}&\varepsilon\zeta&\mathbf{0}\\
 \mathbf{0}&-1&\mathbf{0}\\
 \mathbf{0}&\varepsilon\theta&I_{n-k}
  \end{bmatrix}$ \; where  $\zeta=(b_{1,k},b_{2,k},\dots,b_{k-1,k})^T$, $\theta=(b_{k+1,k},b_{k+2,k},\dots,b_{n,k})^T$, and $\varepsilon \in \{0,1\}$.
\end{enumerate}
\end{Proposition}
\begin{proof} (i)$\Rightarrow$(ii):\; Suppose that $Q$ and $\mu_k(Q)$ share the same exchange polynomial, but $k$ is neither a sink nor a source in $Q$. Since $Q$ does not have $3$-cycles, by the definition of mutation of cluster quivers, when we mutate $Q$ at $k$, the multiplicities of arrows between any two vertices either increase or keep intact. And because $k$ is not a sink or source, there precisely exists some arrow whose multiplicity increases. Then the sum of all of principal minors of order two will changed after mutating at $k$, thus by Lemma 4.4 the exchange polynomial will changed, which is a contradiction.

(ii)$\Rightarrow$(i):\; It follows from Lemma 4.2.

(ii)$\Leftrightarrow$(iii):\; $k\in Q_0$ is a source if and only if $b_{ik}\leqslant 0$ for any $i\in Q_0$; and $k$ is a sink if and only if $b_{jk}\geqslant 0$ for any $j\in Q_0$. Then it follows through comparing the definition of $W$ in (\ref{pmatrix}) and (iii).
\end{proof}

The following conjecture asserts that cospectral cluster quivers form a finite connected subgraph of the exchange graph(see \cite{FZ4}), and one of them can be obtained by mutation at sinks and sources from the other.
\begin{Conjecture}
Let $Q$ be a cluster quiver. Then $Q',Q'' \in Mut(Q)$ are cospectral if and only if there exists a cluster quiver $R \in Mut(Q)$ such that $Q'$ and $R$ are isomorphic and $Q''$ can be obtained from $R$ by mutation at sinks and sources.
\end{Conjecture}

For any seed $\Sigma=(\mathbf{x},\mathbf{y},Q)$ with $Q$ a cluster quiver, let $S(\Sigma) = \bigcup_{\mathbf{x'}\thicksim_c \mathbf{x}}\mathbf{x'}$. We call the subalgebra of the cluster algebra $\mathcal{A}(\Sigma)$ generated by $S(\Sigma)$ a {\bf cospectral subalgebra} corresponding to $\Sigma$, written as $\mathcal{A}_c(\Sigma)$. If $\mathcal{A}(\Sigma)=\mathcal{A}_c(\Sigma)$, we say this cluster algebra $\mathcal{A}(\Sigma)$ to be a {\bf cospectral cluster algebra}.

Clearly, $0\neq \mathbb{ZP}[\mathbf{x}]\subseteqq \mathcal{A}_c(\Sigma)\subseteqq \mathcal{A}(\Sigma)\subseteqq F.$

Let $M(\Sigma)$ denote the set of all seeds mutation equivalent to the seed $\Sigma$. Cospectral relation $\thicksim_c$ for seeds in $M(\Sigma)$ is an equivalence relation whose equivalence class for a seed $\Sigma'$ is denoted by $[\Sigma']$, then we have \[ \mathcal{A}(\Sigma)=\sum_{[\Sigma']\in M(\Sigma)/\thicksim_c} \mathcal{A}_c(\Sigma').\]
\begin{Example} Some examples of cospectral subalgebra are given as follows.
\begin{enumerate}
\item[(i)] Let $\Sigma=(\mathbf{x},\mathbf{y},Q)$ be a seed and $Q$ is a cluster quiver whose underlying graph is $A_3$. Then any cluster quiver in $Mut(Q)$ is either an oriented 3-cycle or a quiver whose underlying graph is $A_3$ (see Lemma 4.11). Since all orientations of $A_3$ are cospectral (see Corollary 3.6) and exchange matrices of oriented 3-cycles are similar for they differ by a permutation, there are exactly two cospectral equivalence classes in $M(\Sigma)$. Let $\Sigma'=(\mathbf{x'},\mathbf{y'},Q')$ be a seed in $M(\Sigma)$ such that $Q'$ is a 3-cycle, then we have $\mathcal{A}(\Sigma) = \mathcal{A}_c(\Sigma) + \mathcal{A}_c(\Sigma')$.
\item[(ii)] Any cluster algebra of rank two associated with a seed whose matrix is skew-symmetric is a cospectral cluster algebra.
\end{enumerate}
\end{Example}

\subsection{Bounds of exchange spectrum radii of cluster quivers}.

In the rest of this section, we consider the bounds of exchange spectrum radii of all cluster quivers in a mutation class. Recall that Fomin and Zelevinsky introduced 2-finite matrices to study finite type classification of cluster algebras, see \cite{FZ2}. For our purposes, we just consider skew-symmetric matrices. For an integer skew-symmetric matrices $B$, $B$ is said to be {\bf 2-finite} if any matrix $B'=(b'_{ij})_{n\times n}$ mutation equivalent to $B$ satisfies that $|b'_{ij}b'_{ji}|\leqslant 3$ for any $i,j \in [1,n]$. Equivalently, any cluster quiver $Q'$ mutation equivalent to the cluster quiver $Q(B)$ is simply-laced.
\begin{Definition}
A valued cluster quiver $Q$ is called {\bf $r$-maximal}$(r > 0)$ if any cluster quiver $Q'$ mutation equivalent to $Q$ has exchange spectrum radius no more than $r$.
\end{Definition}
Note that a cluster quiver is $r$-maximal if and only if so are all of its connected components. It follows from Corollary 3.9 that any full subquiver of a $r$-maximal cluster quiver is $r$-maximal and any cluster quiver contains a full subquiver which is not $r$-maximal is not $r$-maximal.

The following lemmas are well-known.
\begin{Lemma}[\cite{FZ2}]
All orientations of $A_n$ (respectively, $D_n$, $E_6$, $E_7$, or $E_8$) are mutation equivalent.
\end{Lemma}
By Lemma 4.9, we  use $Mut(A_n)$ to denote the mutation class of any cluster quivers whose underlying graphs are $A_n$.
\begin{Lemma}[\cite{FZ2}]
Any $2$-finite connected cluster quiver is mutation equivalent to an orientation of a Dynkin diagram.
\end{Lemma}
\begin{Lemma}[\cite{BV}]
Let $Mut(A_p)$ be the mutation class of $A_p$. Then the class consists of connected quivers satisfying that:
\begin{enumerate}
\item[(i)] All nontrivial cycles are oriented 3-cycles.
\item[(ii)] The degree of any vertex is less than five.
\item[(iii)] If a vertex has degree four, then two of its adjacent arrows belong to one $3$-cycle, and the other two belong to another $3$-cycle.
\item[(iv)] If a vertex has degree three, then two of its adjacent arrows belong to a $3$-cycle, and the third arrow does not belong to any $3$-cycle.
\end{enumerate}
\end{Lemma}
Note that a cycle in the first condition means a cycle in the underlying graph, not passing through the same vertex twice.

The following lemma is a simple observation for the case of cluster quivers whose underlying graphs contain  no 4-cycles.
\begin{Lemma}
Let $B(Q)$ be the exchange matrix of a simply-laced cluster quiver $Q$ with $|Q_0|\geqslant 4$. If the underlying graph $\bar Q$ of $Q$ contains no 4-cycles, then the sum of all principal minors of $B(Q)$ of order four equals to the number of pairs of disadjacent arrows (i.e. without common vertices).
\end{Lemma}
\begin{proof} The principal minor of $B(Q)$ of order four equals to the determinant of the exchange matrix of its corresponding full subquiver. Let us compute the determinant of the exchange matrix $R=(r_{ij})_{4\times 4}$ of a full subquiver $Q'$ of order four. Write $det(R) = \sum_{\pi}\mathrm{sgn}(\pi)r_{1\pi 1}r_{2\pi 2}r_{3\pi 3}r_{4\pi 4}$.

 Since the underlying graph $\bar Q$ of $Q$ contains no 4-cycles, so does the underlying graph of $Q'$. If the term $\mathrm{sgn}(\pi)r_{1\pi 1}r_{2\pi 2}r_{3\pi 3}r_{4\pi 4}$ is not zero, $\pi$ must be a composition of two disjoint 2-cycles and $\mathrm{sgn}(\pi)r_{1\pi 1}r_{2\pi 2}r_{3\pi 3}r_{4\pi 4}=1$. Since each nonzero term corresponds to a pair of disadjacent arrows in a full subquiver of order $4$ in $Q$, thus the sum of all principal minors of $B(Q)$ of order four equals to the number of pairs of disadjacent arrows in $Q$.
\end{proof}

\begin{Remark}
Let $Q$ be a cluster quiver of order 4, and the underlying graph of $Q$ contains no 4-cycles. It is easy to see the determinant of $B(Q)$ does not depend on the orientations of $\bar Q$ from the proof of Lemma 4.12. Therefore its exchange polynomial just depends on its underlying graph. Since it follows from Lemma 4.4, the exchange polynomial of any valued cluster quiver of order 3 does not depend on the orientations, then it is easy to compute the exchange polynomials of cluster quivers of order less than 5.
\end{Remark}

\begin{Theorem}
A connected cluster quiver is 2-maximal if and only if it is mutation equivalent to an orientation of one of $X_2$, $A_1$, $A_2$, $A_3$, or $A_4$, where $X_2$ is a graph with two vertices and two edges.
\end{Theorem}
\begin{proof} By  Lemma 4.11, the underlying graph of any quiver $Q' \in Mut(A_4)$ must be one of graphs in Figure \ref{G4}. These two underlying graphs do not have $4$-cycles, we may compute the exchange radii by using any orientation of them by Lemma 4.12 and Remark 4.13. In any case, it is not difficult to know the exchange spectrum radius is not more than two.
\begin{figure}
\centering
\includegraphics[height=2cm,width=10cm]{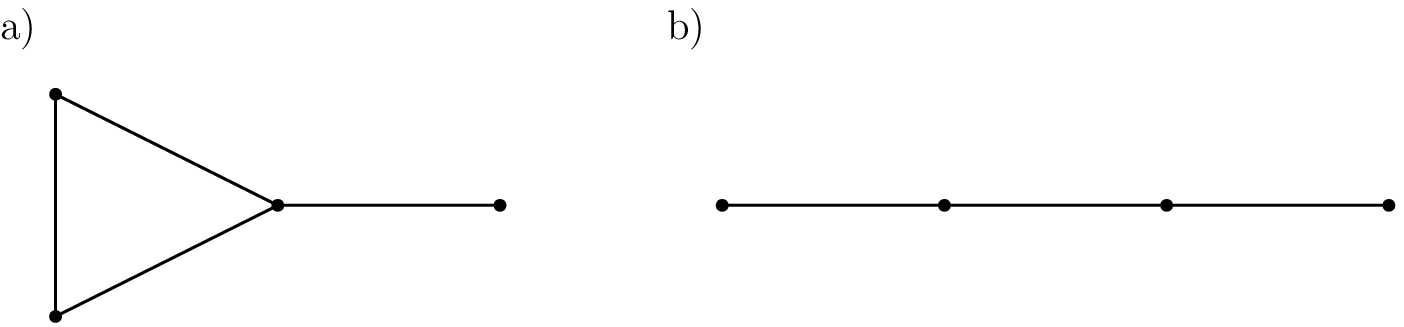}
\caption{}
\label{G4}
\end{figure}

Let $Q$ be a connected $2$-maximal cluster quiver. The multiplicities of arrows must be less than three, otherwise there will be a full subquiver whose exchange spectrum radius is more than two. If there exist arrows, whose multiplicities equal to two in $Q$ and $\bar{Q}$, is not $X_2$, then there exists a full subquiver of $Q$ whose underlying graph is one of the five graphs in Figure \ref{G5}.
\begin{figure}
\centering
\includegraphics[height=4cm,width=10cm]{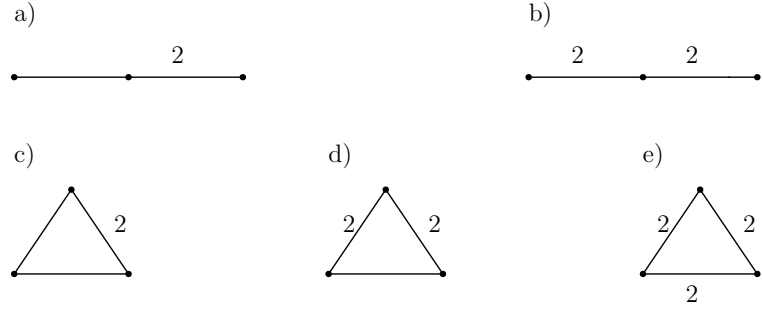}
\caption{The number beside an edge means the multiplicity of the edge}
\label{G5}
\end{figure}
In any case, the exchange spectrum radius of this full valued subquiver is more than 2. Thus $Q$ must be mutation equivalent to an orientation of $X_2$.

Now we suppose that $Q$ is a simply-laced $2$-maximal cluster quiver. Since $Q$ is 2-maximal and connected, any quiver $Q'\sim Q$ must be a simply-laced cluster quiver. Hence $Q$ is 2-finite. By Lemma 4.10, $Q$ is mutation equivalent to an orientation of one of Dinkin diagrams. Since all orientations of a Dynkin diagram are mutation equivalent and share the same exchange polynomial by Lemma 4.9 and Corollary 3.6, respectively. Let us consider the cluster quiver $Q_4$ in Figure \ref{Q4}(a) whose underlying graph is $D_4$.
\begin{figure}
\centering
\includegraphics[height=2.3cm,width=10cm]{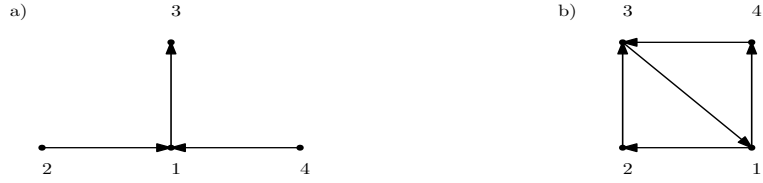}
\caption{$Q_4$ and $\mu_1(Q_4)$}
\label{Q4}
\end{figure}
The exchange spectrum radius of $\mu_1(Q_4)$(see Figure \ref{Q4}(b)) is $\sqrt{5}$ which is more than two. Since $D_n(n \geq 4)$, $E_6$, $E_7$, and $E_8$ contains $D_4$ as an induced subgraph, it follows that $Q$ cannot be mutation equivalent to any orientation of one of $D_n(n \geq 4)$, $E_6$, $E_7$, or $E_8$.

Finally we consider the quiver $Q_5$ in Figure \ref{Q5}(a) whose underlying graph is $A_5$.
\begin{figure}
\centering
\includegraphics[height=2.3cm,width=12cm]{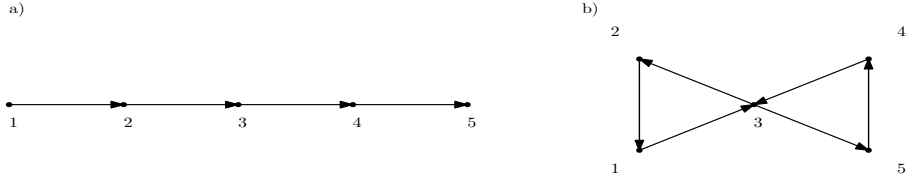}
\caption{$Q_5$ and $\mu_2\mu_4(Q_5)$}
\label{Q5}
\end{figure}
Mutate the quiver at the vertex 2 after mutating at the vertex 4, we get a quiver $\mu_2\mu_4(Q_5)$(see Figure \ref{Q5}(b)) whose exchange spectrum radius is $\sqrt{5}$. In summary, we prove the conclusion.
\end{proof}

We recall preprojective algebras following from \cite{DR}. Let $Q$ be a cluster quiver and \~{Q} be a quiver obtained from $Q$ by adjoining an arrow $\sigma(\alpha): j\rightarrow i$ for each arrow $\alpha : j\rightarrow i$. The preprojective algebra $\Theta(Q)$ of $Q$ is the quotient of the path algebra of \~{Q} modulo the ideal generated by the elements \[ \sum_{t(\beta)=i}\sigma(\beta)\beta,\quad i \in \tilde{Q}_0. \]Then we have the following result.

\begin{Corollary}
Let $Q$ be a cluster quiver whose underlying graph is one of Dynkin diagrams. Then the preprojective algebra $\Theta(Q)$ of $Q$ is representation-finite if and only if $Q$ is 2-maximal.
\end{Corollary}
\begin{proof}
It follows from \cite{DR} that $\Theta(Q)$ is representation-finite if and only if $\bar Q$ is of type $A_1$, $A_2$, $A_3$ or $A_4$. Thus the conclusion follows from Theorem 4.14.
\end{proof}

{\bf Acknowledgements:}\; This project is supported by the National Natural Science Foundation of China (No.11671350 and No.11571173).

\end{document}